\documentclass{birkau}

\usepackage{amsmath,amssymb,url}
\usepackage{enumitem} 

\numberwithin{equation}{section}

\theoremstyle{plain}
\newtheorem{theorem}{Theorem}[section]
\newtheorem{lemma}[theorem]{Lemma}
\newtheorem{proposition}[theorem]{Proposition}
\newtheorem{corollary}[theorem]{Corollary}
\newtheorem{example}[theorem]{Example}

\theoremstyle{definition}

\newtheorem{remark}[theorem]{Remark}

\begin{document}


\title[Transfer Properties in Truncated Vector Lattices]{On the Transfer of Completeness and Projection Properties in Truncated Vector Lattices}



\corrauthor[M. Habibi]{Mohamed Habibi}
\address{Preparatory Institute for Scientific and Technical Studies\\
Carthage University\\2075 Tunis\\Tunisia}
\email{mohamed.habibi@ipest.ucar.tn}

\author[H. Hafsi]{Hamza Hafsi}
\address{Preparatory Institute for Engineering Studies of Tunis\\ University of Tunis: 2 \\1089 Tunis\\Tunisia}
\email{hamza.hafsi@ipeit.rnu.tn}



\begin{abstract}
In this work we investigate the transfer of fundamental order and completeness properties
between truncated Riesz spaces and their unitizations. Specifically, we
provide characterizations and equivalences for several notions of
completeness: the Archimedean property, relatively uniform completeness,
Dedekind completeness, lateral completeness, universal completeness, and the
projection property. Counterexamples are presented to illustrate the necessity
of assumptions and the independence of various completeness notions.
\end{abstract}

\maketitle


\section{Introduction}
The notion of truncation was first introduced by Ball \cite{Ball1,Ball2} in the
context of divisible abelian lattice-ordered groups. A
truncation on an group $A$ is defined as a unary operation $\bar{\cdot
}: A^{+} \rightarrow A^{+}$ satisfying the following axioms for all $a,b \in
A^{+}$:
\begin{itemize}
	\item[$(\tau_{1})$] $a \wedge\bar{b} \leq\bar{a} \leq a$, 
	\item[$(\tau_{2})$] $\bar{a} = 0$ implies $a = 0$, 
	\item[$(\tau_{3})$] If $\overline{na} = na$ for all $n \in\mathbb{N}$, then $a
	= 0$.
\end{itemize}
This framework was later extended to Riesz spaces in \cite{bouhafsimahfoudhi}, where a
truncation is a map from the positive cone $E^{+}$ of a Riesz space $E$ into
itself, satisfying at least $(\tau_{1})$ and $(\tau_{2})$. A Riesz space
equipped with such a map is called a \emph{truncated Riesz space}.
A truncation satisfying the additional axiom $(\tau_{3})$ is called an
\emph{Archimedean truncation}. Interestingly, the Archimedean property of a
truncation and that of the underlying Riesz space are independent concepts. In
fact, on the one hand, if we consider the plan $E=%
\mathbb{R}
^{2},$ equipped with its lexicographic ordering, it's well known that $E$ is
not Archimedean Riesz space, yet the truncation defined on $E$ by
$ 
\overline{\left(  x,y\right)  }=\left(  x,y\right)  \wedge\left(  0,1\right)
\text{, for all }\left(  x,y\right)  \in E^{+}%
$ 
is an Archimedean truncation on $E$. But also, on the other hand, if
$E=\{0\}\times\mathbb{R}$, then \textit{obviously}, $E$ \textit{is an
}Archimedean vector sublattice of $%
\mathbb{R}
^{2}$. 

Define a truncat\'{\i}on on $E$ by setting:%
\[
\overline{(0,r)}=(0,\ r)\wedge(1,0)=(0,\ r)\text{ for all }r\in\mathbb{R}%
\]
(\textit{the infimum is taken in} $\mathbb{R}^{2}$). \textit{Observe that if}
$n\in\{1,2,...\}$,\textit{then }%
\[
\overline{n(0,r)}=n(0,\ r)\text{ for all }r\in\mathbb{R}.
\]
Hence this truncation is not an Archimedean truncation on $E.$
Furthermore, if there exists $u\in E^{+}$ such that for all $x\in E^{+}$, we
have $\bar{x}=x\wedge u$, then the truncation is called unital, and the Riesz
space $E$ is unital. This element $u$ is referred to as a truncation unit of
$E$. Clearly, a unital Riesz space has only one truncation unit, and condition
$\left(  \tau_{2}\right)  $ implies that the truncation unit $u$ of $E$ is a
weak order unit. It should be noted that not any truncation is unital. Take for example $E=c_{00}$ the Riesz space of eventually zero real sequences
equipped with the pointwise order and for $u=\left(  u_{n}\right)  \in
c_{00},$ let $\bar{u}=u\wedge1,$ where $1$ is the constant sequence equal $1$.

Another example illustrating this situation : Let $X$ be a locally compact
Hausdorff space. Recall that a real- valued function $f$ on $X$ is said to
vanish at infinity if, for every $\epsilon\in\left(  0,+\infty\right)  $ the
subset $\left\{  x\in X,\left\vert f\left(  x\right)  \right\vert \geq
\epsilon\right\}  $ is a compact set in $X.$ The subset $C_{0}\left(
X\right)  $ of all such fonction is a Riesz subset of $C\left(  X\right)  $
and the formula $\bar{f}\left(  x\right)  =\min\left(  1,f\left(  x\right)
\right)  $ for all $f\in C_{0}\left(  X\right)  ^{+}$ defines a truncation
which is unital if and only if $X$ is compact.

If $E$ is a truncated Riesz space, the set of all fixed points is defined by
$\bar{E}=\left\{ x\in E,\overline{\mid x\mid }=\mid x\mid \right\} $. This set will play a
fundamental role in our study.  Note,  in passing,  that $0\in\overline{E}$. 

The notion of unitization of a truncated vector lattice was introduced in
\cite{bouhafsimahfoudhi} as follows: A unital vector lattice $\epsilon E$ with truncation unit
$u\geq0$ is called a unitization of a truncated vector lattice $E$ if $E$ is a
vector sublattice of $\epsilon E$ and $\bar{x}=x\wedge u,$ for all $x\in
E^{+}.$ Obviously, any unital truncated Riesz space is its own unitization. It
was proved in \cite{BoulaChiheb} that if we identify $E$ with $E\times\left\{
0\right\}  $ and $%
\mathbb{R}
$ with $\left\{  0\right\}  \times%
\mathbb{R}
$ both in $E\times%
\mathbb{R}
,$ then $E\oplus%
\mathbb{R}
=\left\{  x+\lambda,x\in E,\lambda\in%
\mathbb{R}
\right\}  $ can be endowed with the structure of Riesz space with positive
cone defined as%
\[
\left(  E\oplus%
\mathbb{R}
\right)  ^{+}=E^{+}\cup\left\{  x+\lambda;0<\lambda\in%
\mathbb{R}
\text{, }x\in E\text{ and }\frac{1}{\lambda}x^{-}\in\bar{E}\right\}  .
\]
The absolute value of an element $x+\lambda\in E\oplus%
\mathbb{R}
$ is given by%
\[
\left\vert x+\lambda\right\vert =\left\{
\begin{array}
	[c]{l}%
	\left\vert x\right\vert -2\left\vert \lambda\right\vert \overline{\left(
		\frac{1}{\lambda}x^{-}\vee-\frac{1}{\lambda}x^{+}\right)  }+\left\vert
	\lambda\right\vert \text{ if }\lambda\neq0\\
	\left\vert x\right\vert \text{ if }\lambda=0
\end{array}
\right.
\]
It turns out that $E\oplus%
\mathbb{R}
$ is a unitization of $E$ and the truncation on $E\oplus%
\mathbb{R}
$ is defined by met of $1.$
\bigskip To be more precise regarding the relationship between $E$ and
$E\oplus%
\mathbb{R}
$,  the following theorem was established in \cite{bouhafsimahfoudhi}:
\begin{theorem}
	Let $E$ be a truncated Riesz space and $E^{\perp}$ its disjoint complement in
	$E\oplus%
	\mathbb{R}
	$. The following assertions hold
	\begin{enumerate}
		\item $\bar{E}=\left\{ x\in E:\mid x\mid \leq 1\text{ in }E\oplus 
		\mathbb{R}
		\right\} .$
		\item $E$ is a maximal ideal in $E\oplus%
		\mathbb{R}
		.$
		\item If $E$ is not unital then $E^{\perp}=\left\{  0\right\}  .$
		\item If $E$ is unital with truncation unit $u,$ then
		\[
		E^{\perp}=%
		\mathbb{R}
		\left(  1-u\right)  =\left\{  \lambda\left(  1-u\right)  ,\lambda\in%
		\mathbb{R}
		\right\}  .
		\]
		\item $E$ is dense in $E\oplus%
		\mathbb{R}
		$ if and only if, $E$ is not unital.
	\end{enumerate}
\end{theorem}
\section{Preliminary Results}
In this section, we will gather the basic properties of truncated Riesz
spaces. Some of them can be found in \cite{Boulstructuretheorem}.  We include detailed proofs for
clarity and convenience.
\begin{proposition}
	Let $\bar{\cdot}$ be a truncation on $E$. The following conditions are equivalent:
	\begin{enumerate}
		\item[(i)] For all $a,b\in E^{+}$, we have $a\wedge\bar{b}\leq\bar{a}\leq a $.
		
		\item[(ii)] For all $a,b\in E^{+}$, we have $a\wedge\bar{b}=\bar{a}\wedge b$.
	\end{enumerate}
\end{proposition}
\begin{proof}
	(i) $\Rightarrow$ (ii) Choose $a,b\in E^{+}$and observe that%
	\[
	a\wedge\bar{b}\leq\bar{a}\text{ and }a\wedge\bar{b}\leq\bar{b}\leq b
	\]
	We derive that%
	\[
	a\wedge\bar{b}\leq\bar{a}\wedge b
	\]
	and so $a\wedge\bar{b}=\bar{a}\wedge b$, as desired.
	
	(ii) $\Rightarrow$ (i) Pick $a,b\in E^{+}$. We have%
	\[
	a\wedge\bar{b}=\bar{a}\wedge b\leq\bar{a}=\bar{a}\wedge\bar{a}=a\wedge
	\overline{\bar{a}}\leq a
	\]
	This completes the argument.
\end{proof}
\bigskip
The following properties are immediate consequences of the definition of a truncation:
\begin{proposition}
	Let $E$ be a truncated Riesz space. For all $x,y \in E^{+}$, the following hold:
	\begin{enumerate}
		\item $\bar{x}\leq x.$
		\item $x\leq y$ implies $\bar{x}\leq\bar{y}$;
		\item $\bar{\bar{x}}=\bar{x}$ (idempotency);
		\item $x\in\bar{E}$ if and only if $x=\bar{y}$ for some $y\in E^{+}$;
		\item If $x\leq y$ and $y\in\bar{E}$, then $x\in\bar{E}$;
		\item (Birkhoff's Inequality) $\left\vert \bar{x}-\bar{y}\right\vert
		\leq\overline{\left\vert x-y\right\vert }$.
	\end{enumerate}
\end{proposition}
\begin{proof}
	For all $x,y\in E^{+}$
	\begin{enumerate}
		\item We have that, $\overline{x}=\overline{x}\wedge\bar{x}=x\wedge\bar
		{\bar{x}}\leq x.$
		\item If $x\leq y$ implies $\overline{x}=\overline{x}\wedge x\leq\overline
		{x}\wedge y=x\wedge\overline{y}\leq\overline{y}$
		\item we have $\bar{\bar{x}}\leq\bar{x}$ and $\bar{x}=\bar
		{x}\wedge\bar{x}=\bar{\bar{x}}\wedge x\leq\bar{\bar{x}}\leq \bar{x}$ ,then $\bar{\bar{x}%
		}=\bar{x}$
		\item If $x\in\bar{E},$ then $x=\bar{x}.$
		The converse follows directly from the definition of $\bar{E}$ and the first inequality.
		\item We have that $x=x\wedge y=x\wedge\bar{y}=\bar{x}\wedge y\leq\bar{x}\leq x$. Hence $\bar{x}=x$ and $x\in\bar{E}$.
		\item Using the standart Birkhoff's Inequality ,\cite[Theorem 1.9]{positiveoperator}, we have
		that
		\begin{align*}
			|\bar{x}-\bar{y}|  & =|\bar{x}\wedge\left(  x\vee y\right)  -\bar{y}%
			\wedge\left(  x\vee y\right)  |\\
			& =|x\wedge\overline{\left(  x\vee y\right)  }-y\wedge\overline{\left(  x\vee
				y\right)  }|\\
			& \leq|x-y|.
		\end{align*}
		Moreover, using $\left(  1\right)  ,$ we get%
		\[
		\bar{x}-\bar{y}\leq\bar{x}\leq\overline{x+y}%
		\]
		and similary, $\bar{y}-\bar{x}\leq\overline{x+y}.$ Then, we have that%
		\begin{align*}
			|\bar{x}-\bar{y}|  & \leq\overline{\left(  x+y\right)  }\wedge|x-y|\\
			& =\left(  x+y\right)  \wedge\overline{|x-y|}\\
			& \leq\overline{|x-y|}.
		\end{align*}
	\end{enumerate}
\end{proof}
The set of all fixed points determines the truncation entirely, as stated in
the following result:
\begin{lemma}
	Two truncations on the same Riesz space $E$ are equal if and only if they have
	the same set of fixed points.
\end{lemma}
\begin{proof}
	Consider two truncations $x\mapsto x^{\ast}$ and $x\mapsto x^{\star}$ on a Riesz
	space $E$ and let $E_{\ast}=\{x\in E:\mid x\mid^{\ast}=\mid x\mid\},E_{\star}=\{x\in
	E:\mid x\mid^{\star}=\mid x\mid\}.$
	
	The \textquotedblleft only if\textquotedblright\ direction is obvious. For the
	\textquotedblleft if\textquotedblright\ part, suppose that $E_{\ast}=E_{\star
	}$and let $x\in E^{+}$. Since $x^{\ast}\leq x$, we have $(x^{\ast})^{\star
	}\leq x^\star$.
	But $x^{\ast}\in E_\ast= E_\star$,  hence $(x^{\ast})^{\star}=x^{\ast},$ and so
	$x^{\ast}\leq x^\star$.
	Similarly,  $x^{\star}\leq x^{\ast}$, which implies $x^{\ast}=x^{\star}$.  This
	proves the proposition.
\end{proof}
\section{Archimedean Property}
In this section, we examine necessary and sufficient conditions under which
the unitization $E \oplus\mathbb{R}$ of a truncated Riesz space $E$ is Archimedean.
It is straightforward to observe that if $E \oplus\mathbb{R}$ is Archimedean,
then the element $1$ acts as a strong order unit and the truncation satisfies
condition $(\tau_{3})$ introduced by Ball \cite{Ball1}. The following result provides a
complete characterization of this situation.
\begin{theorem}
	Let $E$ be a truncated Riesz space. Then $E \oplus\mathbb{R}$ is Archimedean
	if and only if $E$ is Archimedean and the truncation satisfies $(\tau_{3})$.
\end{theorem}
\begin{proof}
	Assume that $E$ is Archimedean and that $(\tau_{3})$ holds. Let $x,y\in E$ and
	$\lambda,\beta\in\mathbb{R}$ such that
	\[
	0\leq n(x+\lambda)\leq y+\beta\quad\text{for all }n\in\mathbb{N}.
	\]
	This implies that $\lambda\geq0$ and
	\[
	0\leq y-nx+\beta-n\lambda\quad\text{for all }n.
	\]
	In particular, it follows that $0\leq\beta-n\lambda$ for all $n$ . Then,
	$\lambda=0$ and $0\leq nx\leq y+\beta$ for all $n$. If $\beta=0,$ by the
	Archimedean property of $E$, we conclude that $x=0$, and hence $x+\lambda=0$.
	If $\beta>0$, fix $n\in\mathbb{N}^{\ast}$ and consider $m\geq 1$.  It follows that
	$$0\leq nmx\leq y+\beta \leq y+m\beta .$$
	Thus, we obtain the following inequality
	\[
	0\leq\frac{nm}{\beta}x\leq\frac{1}{\beta}y+m.
	\]
	This yields%
	\[
	m\left(  \frac{n}{\beta}x-1\right)  ^{+}\leq\frac{1}{\beta}y^{+}.
	\]
	But%
	\begin{align*}
		\frac{n}{\beta}x-\overline{\frac{n}{\beta}x}  & =\frac{n}{\beta}x-\frac
		{n}{\beta}x\wedge1\\
		& =0\vee\left(  \frac{n}{\beta}x-1\right)  \\
		& =\left(  \frac{n}{\beta}x-1\right)  ^{+}%
	\end{align*}
	Then
	\[
	m\left(  \frac{n}{\beta}x-\overline{\frac{n}{\beta}x}\right)  ^{+}\leq\frac
	{1}{\beta}y^{+}.
	\]
	Since this inequality holds for all $m$, the Archimedean property of $E$
	implies
	\[
	\frac{n}{\beta}x=\overline{\frac{n}{\beta}x}.
	\]
	Then by $(\tau_{3})$, we obtain $x=0$. Thus, $x+\lambda=0$, proving that
	$E\oplus\mathbb{R}$ is Archimedean.
	Conversely, assume that $E \oplus\mathbb{R}$ is Archimedean. Since $E$ is a
	vector sublattice of $E \oplus\mathbb{R}$, it is also Archimedean. Now, let $x
	\in E^{+}$ be such that $\overline{nx} = nx$ for all $n$. Then $nx \leq1$ in
	$E \oplus\mathbb{R}$ for all $n$, and since this space is Archimedean, it
	follows that $x = 0$. Hence $(\tau_{3})$ holds.
\end{proof}
The result above leads to a simpler characterization in the unital case.
\begin{corollary}
	Let $E$ be a unital truncated Riesz space. Then $E\oplus\mathbb{R}$ is
	Archimedean if and only if $E$ is Archimedean.
	\begin{proof}
		The forward implication is immediate. Conversely, assume that $E$ is
		Archimedean and let $u$ be its truncation unit. If $x\in E^{+}$ satisfies
		$\overline{nx}=nx$ for all $n,$ then%
		\[
		nx=nx\wedge u\leq u\text{ for all }n
		\]
		This implies $x=0$ by the Archimedean property of $E,$ so $\left(  \tau
		_{3}\right)  $ hold. Then, by the previous theorem $E\oplus\mathbb{R}$ is
		Archimedean. 
	\end{proof}
\end{corollary}
The following theorem relates the set of fixed points of the truncation to its
unit element, thus concluding this section.
\begin{theorem}
	Let $E$ be a non unital truncated Riesz space such that $E\oplus%
	\mathbb{R}
	$ is Archimedean.
	For every $x\in E\oplus%
	\mathbb{R}
	$ with $x>0,$ $\sup\left\{  \bar{y},y\in E,0\leq y\leq x\right\}  =\bar{x}$,
	(where the supremum is taken in $E\oplus%
	\mathbb{R}
	).$ In particular,  $\sup\left\{  \bar{x},x\in E^{+}\right\}  =1.$
\end{theorem}
\begin{proof}
	It's clear that $\bar{x}$ is an upper bound of the set $\left\{  \bar{y},y\in
	E,0\leq y\leq x\right\}  .$ Assume by contradiction that some $z\in E\oplus%
	\mathbb{R}
	$ satisfies $z<\bar{x}$ and $\bar{y}\leq z$ for all $y\in E^{+}$ such that
	$y\leq x.$ By the order denseness of $E$ in $E\oplus%
	\mathbb{R}
	$, there exist $u\in E$ with $0<u\leq\bar{x}-z\leq\bar{x}\leq x.$ Therefore,
	$\bar{u}\leq z$ and $0<2\bar{u}=\bar{u}+\bar{u}\leq\bar{x}-z+z.$ By induction
	, $0<n\bar{u}\leq\bar{x}$ for all $n\in%
	\mathbb{N}
	,$ contradicting the Archimedean property of $E\oplus%
	\mathbb{R}
	.$
	Moreover,%
	\begin{align*}
		1  & =\sup\left\{  \bar{y},y\in E,0\leq y\leq1\right\}  \\
		& =\sup\left(  \left[  0,1\right]  \cap E^{+}\right)  \\
		& =\sup\left\{  \bar{x},x\in E^{+}\right\}  .
	\end{align*}
	This completes the proof.
\end{proof}
\begin{remark}
	Similarly, an analogous result can be obtained in the unital case as follow:
	If $0<x=x_{1}+\lambda\left(  1-u\right)  \in E\oplus%
	\mathbb{R}
	\left(  1-u\right)  ,$ then%
	\begin{align*}
		\sup\left\{  \bar{y},y\in E,0\leq y\leq x\right\}    & =\left\{  \bar{y},y\in
		E,0\leq y\leq x_{1}\right\}  =\overline{x_{1}}\\
		& =x_{1}\wedge u\\
		& =x_{1}\wedge u+\lambda\left(  1-u\right)  \wedge u\text{ (since }%
		u\wedge\left(  1-u\right)  =0)\\
		& =x\wedge u
	\end{align*}
\end{remark}
\section{Relatively Uniformly Complete Property}
	In this section, we investigate the relatively uniformly completeness of the
	unitization $E \oplus\mathbb{R}$ of a truncated Riesz space $E$.
	
	Given $u\in E^{+}$,  it is said that the sequence $\left( x_{n}\right) $ in $E$ converges $u$-uniformly to the element $x\in E$, or that $x$ is an $u$-uniform limit for $\left( x_{n}\right) $,  whenever, for every $ \epsilon >0 $,  there exists a natural number $N_{\epsilon}$ such that $\mid x_{n}-x\mid \leq \epsilon u$ holds for all $n\geq N_{\epsilon}$.  It is said that the sequence $\left( x_{n}\right) $ in $E$ converges relatively uniformly to $x$ whenever $x_{n}$ converges $u$-uniformly to $x$ for some $u\in E^{+}$.
The sequence $\left(x_{n}\right) $ in $E$ is an $u$-uniform Cauchy sequence with $u>0$ if, for every $\epsilon >0$, we have $\mid x_{n}-x_{m}\mid \leq \epsilon u$ for all $n$ and $m$ sufficiently large.  The Riesz space $E$ is called uniformly complete whenever, for every $u\in E^{+}$, every $u$-uniform Cauchy sequence has an $u$-uniform limit. 
For further information regarding uniformly complete Riesz spaces, one may consult	\cite{Riesz1}.
	We begin with a characterization that connects the relative uniform
	completeness of $E \oplus\mathbb{R}$ to that of $E$ itself.
	\begin{proposition}
		Let $E$ be a truncated Riesz space. If $E \oplus\mathbb{R}$ is relatively
		uniformly complete, then $E$ is relatively uniformly complete.
	\end{proposition}
	\begin{proof}
		Let $u\in E^{+}$ be a positive element. Consider an $u-$Cauchy sequence
		$\left(  x_{n}\right)  _{n\in%
			\mathbb{N}
		}$\ in $E.$ By definition,  this means that there exists $\epsilon >0$ and a positive integer $N_{\epsilon}$ such that for all $n,m\geq N_{\epsilon}$
		$$
		\mid x_{n}-x_{m}\mid \leq \epsilon u
		$$
		We view $\left(  x_{n}\right)  _{n\in%
			\mathbb{N}
		}$ as a sequence in  $E\oplus\mathbb{R}$ .  The space $E\oplus\mathbb{R}$ is
		assumed to be uniformly complete. Therefore, there exists $v=a+\alpha\in\left(
		E\oplus%
		\mathbb{R}
		\right)  ^{+}$ such that for $n\geq N_{\epsilon}
		$%
		\[
		\mid x_{n}-v\mid\leq\epsilon u
		\]
		We consider two cases for $\alpha:$ If $\alpha=0,$  this means the sequence
		converges to $v$ within $E.$ Otherwise,  $\alpha >0$  and for all $n\geq N_{\epsilon}
		$%
		\[
		\mid x_{n}-a\mid-2\alpha\overline{\frac{1}{\alpha}\left(  x_{n}-a\right)
			^{-}\vee\frac{-1}{\alpha}\left(  x_{n}-a\right)  ^{+}}+\alpha\leq\epsilon u
		\]
		Then $-\alpha \geq0.$ This directly contradicts our assumption for this case,  which
		was $\alpha>0.$
		Since case $2$ leads to a contradiction,  the only possibility is case $1$,
		where $\lambda =0$,  and the proof is finished.
	\end{proof}
The converse of the last result is not always valid (an example will be given
later). However,  it holds in the unital case, as stated in the next proposition.
\begin{proposition}
	Let $E$ be a truncated Riesz space. If $E$ is unital and uniformly complete,
	then $E\oplus%
	\mathbb{R}
	$ is uniformly complete.
\end{proposition}

\begin{proof}
	The truncation in $E$ will be defined by met of $u\in E^{+},$ $\bar{%
		x}=x\wedge u$ for all $x\in E^{+}.$
	Consider a sequence $\left( x_{n}\right) _{n\in 
		\mathbb{N}
	}\in E\oplus E^{d},$ where each element $x_{n}$ is represented as $%
	x_{n}=a_{n}+b_{n\text{ }}$ with $a_{n}\in E$ and $b_{n}\in E^{d}.$ Assume $%
	\left( x_{n}\right) _{n\in 
		\mathbb{N}
	}$ is a $v-$Cauchy sequence for some $v=a+\lambda \left( 1-u\right) \in
	E\oplus E^{d}$ \ where $a\in E^{+}$ and $\lambda \in 
	\mathbb{R}
	^{+}.$ By the definition of a $v$-Cauchy sequence,  there exists $\epsilon >0$ such that 
	\begin{equation*}
		\mid x_{n}-x_{m}\mid =\mid a_{n}-a_{m}+b_{n}-b_{m}\mid \leq \epsilon v 
	\end{equation*}
	
	Using the decomposition property \cite[Theorem 1.13]{positiveoperator},  the last inequality implies that the components satisfy the following
	bounds for $n$ and $m$ large enough
	\begin{equation*}
		\mid a_{n}-a_{m}\mid \leq \epsilon  a\text{ and }\mid b_{n}-b_{m}\mid
		\leq \epsilon \lambda \left( 1-u\right) 
	\end{equation*}%
	These inequalities show that $\left( a_{n}\right) _{n\in 
		\mathbb{N}
	}$ is an $a$-Cauchy sequence in $E$ and $\left( b_{n}\right) _{n\in 
		\mathbb{N}
	}$ is a $\lambda (1-u)$-Cauchy sequence in $E^{d}=%
	\mathbb{R}
	\left( 1-u\right) .$ Thus, there exist elements $x\in E$ and $y\in E^{d}$
	such that for all $n$ sufficiently large
	\begin{equation*}
		\mid a_{n}-x\mid \leq \epsilon a\text{ and }\mid b_{n}-y\mid \leq
		\epsilon \lambda \left( 1-u\right) 
	\end{equation*}%
	Therefore,  by combining these inequalities and applying the triangle inequality
	for the absolute value in $E\oplus E^{d}$, we obtain:%
	\begin{equation*}
		\mid x_{n}-(x+y)\mid \leq \epsilon v
	\end{equation*}%
	This demonstrates that the $v$-Cauchy sequence $\left( x_{n}\right) _{n\in 
		\mathbb{N}
	}$ converges $v$-uniformly to $x+y$ in $E\oplus 
	\mathbb{R}
	$
\end{proof}
\begin{example}
	Let $E=c_{00}$ equipped with the pointwise order. For $u=\left( u_{n}\right)
	_{n\in 
		\mathbb{N}
	}\in c_{00},$ let $\bar{u}=u\wedge 1,$ where $1$ is the constant sequence
	equal $1.$ We claim that E is relatively uniformly complete, but $E\oplus 
	\mathbb{R}
	$ is not relatively uniformly complete.
	
	\begin{enumerate}
		\item E is relatively uniformly complete.
		
		Indeed,  consider an arbitrary relatively uniform Cauchy sequence $\left(
		u_{n}\right) _{n\in 
			\mathbb{N}
		}\in E.$  There exists an element $v\in E^{+}$ such that for every $\epsilon>0$ and for all integer $n$ and $m$ large enough 
		$
		,\mid u_{n}-u_{m}\mid \leq \epsilon v.$ Let $u_{n}$  be represented by its
		components $u_{n}=\left( a_{n,k}\right) _{k\in 
			\mathbb{N}
		}$ $\ $and $v=\left( v_{k}\right) _{k\in 
			\mathbb{N}
		}$.  The given inequality is equivalent to  
		\begin{equation*}
			\forall k\in 
			\mathbb{N}
			,\mid a_{n,k}-a_{m,k}\mid \leq \epsilon v_{k}
		\end{equation*}
		
		For each fixed $k$,  the sequence of components $\left( a_{n,k}\right) _{n\in 
			\mathbb{N}
		}$ forms a Cauchy sequence in $%
		\mathbb{R}
		.$ Since $%
		\mathbb{R}
		$ is complete, this sequence converges to a limit, which we denote by $%
		c_{k}\in 
		\mathbb{R}
		$. Let $c=\left( c_{k}\right) _{k\in 
			\mathbb{N}
		}$ be the sequence of these limits.  As $v\in E,$ its support is finite.  This
		means there exists an integer $k_{0}$ such that $v_{k}=0$ for all $k
		>k_{0}.$ Hence,  for $k>k_{0},$ $a_{n,k}=a_{N_{\epsilon},k},$ for all $n\geq N_{\in} $ which implies that $%
		c_{k}=a_{N_{\epsilon},k}=0$ for $k$ large enough. This demonstrates that the limit
		sequence $c$ has finite support, confirming that $c\in E.$ Obviously,  the
		sequence $\left( u_{n}\right) _{n\in 
			\mathbb{N}
		}$ converge $v$-uniformly to $c$ in $E.$
		
		\item $E\oplus 
		\mathbb{R}
		$ is not relatively uniformly complete.
		
		Now, put for all $n\in 
		\mathbb{N}
		,$ $u_{n}=\left( a_{n,k}\right) _{k\in 
			\mathbb{N}
		}$ the sequence in $E$ defined by%
		\begin{equation*}
			a_{n,k}=\left\{ 
			\begin{array}{l}
				\frac{1}{k}\text{ if }0\leq k\leq n \\ 
				0\,\ \text{if }k>n%
			\end{array}%
			\right. 
		\end{equation*}%
		Let $\epsilon >0.$ There exist some $n(\epsilon )\in 
		\mathbb{N}
		$ such that%
		\begin{equation*}
			\frac{1}{k}\leq \epsilon \text{ for all }k\geq n(\epsilon )
		\end{equation*}%
		For all $n\geq m\geq n(\epsilon );$%
		\begin{equation*}
			\overline{\frac{1}{\epsilon }\left\vert u_{n}-u_{m}\right\vert }=\frac{1}{%
				\epsilon }\left\vert u_{n}-u_{m}\right\vert 
		\end{equation*}%
		and consequently, $\left\vert u_{n}-u_{m}\right\vert \leq \epsilon .1.$ We
		deduced that the sequence $(u_{n})_{n\in 
			\mathbb{N}
		}$ is $\ 1$-uniform cauchy in $E\oplus 
		\mathbb{R}
		$. If the sequence $\left( u_{n}\right) _{n\in n}$ converge relatively
		uniformly to $u+\lambda $ in $E\oplus 
		\mathbb{R}
		$ with $u=\left( x_{1},...,x_{n_{0}},0,...\right)\in E ,$ then for $n$ large
		enough%
		\begin{equation*}
			\left\vert u_{n}-u-\lambda \right\vert \leq \epsilon .1.
		\end{equation*}%
		If $\lambda \neq 0,$ then%
		\begin{equation*}
			\left\vert u_{n}-u\right\vert -2\left\vert \lambda \right\vert \overline{%
				\left( \frac{1}{\lambda }\left( u_{n}-u\right) ^{-}\vee -\frac{1}{\lambda }%
				\left( u_{n}-u\right) ^{+}\right) }+\left\vert \lambda \right\vert \leq
			\epsilon .1
		\end{equation*}%
		Which yield to the following contradiction 
		\begin{equation*}
			\forall \epsilon >0,\epsilon \geq \left\vert \lambda \right\vert >0.
		\end{equation*}%
		Hence, $\lambda =0$ and 
		\begin{equation*}
			\forall n\geq n_{1}>n_{0},\left\vert u_{n}-u\right\vert \leq \epsilon .1.
		\end{equation*}%
		So, 
		\begin{equation*}
			\forall n\geq n_{1}>n_{0},\overline{\frac{1}{\epsilon }\left\vert
				u_{n}-u\right\vert }=\frac{1}{\epsilon }\left\vert u_{n}-u\right\vert 
		\end{equation*}%
		Then, 
		\begin{equation*}
			\forall n\geq n_{1}>n_{0},\forall k\in 
			\mathbb{N}
			,\left\vert u_{n,k}-u_{k}\right\vert \leq \epsilon .
		\end{equation*}%
		In particulary,%
		\begin{equation*}
			\frac{1}{n_{0}+1}=\left\vert u_{n_{0}+1,n_{0}+1}-u_{n_{0}+1}\right\vert \leq
			\epsilon 
		\end{equation*}%
		which is a contradiction.
	\end{enumerate}
\end{example}
\section{Dedekind Completeness}
This section is dedicated to examining the Dedekind completeness of the
unitization $E\oplus 
\mathbb{R}
$ of a truncated Riesz space $E$. We recall that a Riesz space is Dedekind
complete if every nonempty subset that is bounded above admits a supremum.

The next lemma characterise Dedekind complete Riesz spaces.
\begin{lemma}\cite[Lemma 1.39]{locallysolid}
	A Riesz space $E$ is Dedekind complete if and only if for every net $\left(
	u_{\alpha }\right) $ in $E$ satisfying $0\leq u_{\alpha }\uparrow \leq v$ in 
	$E$ we have $u_{\alpha }\uparrow u$ for some $u.$
\end{lemma}
We introduce another lemma that will be needed in the proof of the main theorem of this section, which is devoted to adapting the decomposition property in Riesz spaces to the setting of increasing nets.
\begin{lemma}
Let $ \left(x_{\alpha }\right) $ a net  in a Riesz space satisfying $ 0\leq x_{\alpha }\uparrow \leq u+v$ , then there exists a net $\left(u_{\alpha }\right) $ and $\left(v_{\alpha }\right) $ satisfying $x_{\alpha }=u_{\alpha }+v_{\alpha }$ , $0\leq u_{\alpha } \uparrow \leq \mid u\mid$ and $ 0\leq v_{\alpha }\uparrow \leq \mid v\mid $.
\end{lemma}
\begin{proof}
Inspired by the classical proof of the decomposition property in Riesz spaces, let us consider $u_{\alpha }=\left[x_{\alpha } \vee\left(-\left|u\right|\right)\right] \wedge\left|u\right|$, and observe that $0\leq u_{\alpha } \leq\left|u\right|.$ Now put $v_{\alpha }=x_{\alpha }-u_{\alpha }$ and observe that

$$
v_{\alpha }=x_{\alpha }-\left[x_{\alpha } \vee\left(-\left|u\right|\right)\right] \wedge\left|u\right|=\left[0 \wedge\left(x_{\alpha }+\left|u\right|\right)\right] \vee\left(x_{\alpha }-\left|u\right|\right)
$$

On the other hand,  $0\leq  x_{\alpha } \leq\left|u\right|+\left|v\right|$, from which it follows that

$$
0\leq  v_{\alpha } \leq 0 \vee\left(x_{\alpha }-\left|u\right|\right) \leq\left|v\right|
$$

Thus, $0\leq v_{\alpha } \leq\left|v\right|$ also holds.
Since $\left(x_{\alpha }\right)$ is increasing and our attention is focused on expressions $$u_{\alpha }=\left[x_{\alpha } \vee\left(-\left|u\right|\right)\right] \wedge\left|u\right|$$ and $$
v_{\alpha }=\left[0 \wedge\left(x_{\alpha }+\left|u\right|\right)\right] \vee\left(x_{\alpha }-\left|u\right|\right),
$$we can conclude that the two nets  $\left(u_{\alpha }\right) $  and $\left(v_{\alpha }\right)$ are also increasing. This completes the proof.
\end{proof}
In general,  when the suprema exists,  the supremum of the sum of two nets may differ from the sum of their individual suprema. However,  in the case of increasing nets,  the equality holds. This is stated and proved in the next lemma.
\begin{lemma}
Let $(x_\alpha)_{\alpha \in A}$ and $(y_\alpha)_{\alpha \in A}$ be two increasing nets indexed by the same directed set $A$,  and suppose that $\sup x_\alpha$ and $\sup y_\alpha$ exist. Then,  the supremum of the net $(x_\alpha + y_\alpha)_{\alpha \in A}$ exists,  and we have:
\[
\sup_{\alpha \in A}(x_\alpha + y_\alpha) = \sup_{\alpha \in A} x_\alpha + \sup_{\alpha \in A} y_\alpha.
\]
\end{lemma}
\begin{proof}
Let $x := \sup_{\alpha \in A} x_\alpha$ and $y := \sup_{\alpha \in A} y_\alpha$.

Define the net $(z_\alpha)_{\alpha \in A}$ by $z_\alpha := x_\alpha + y_\alpha$.  Since both $(x_\alpha)$ and $(y_\alpha)$ are increasing,  and addition is order-preserving, the net $(z_\alpha)$ is also increasing.\\
For all $\alpha \in A$,  we have:
$z_\alpha = x_\alpha + y_\alpha \leq x + y$ , so $x + y$ is an upper bound for the net $(z_\alpha)$.

Now,  let us show that $x + y$ is the least upper bound of the net $(z_\alpha)_{\alpha \in A}$.
Suppose that \( u \) is any upper bound of the net \( (z_\alpha) \),  where \( z_\alpha = x_\alpha + y_\alpha \).  Let \( \alpha, \beta \in A \).  Since \( A \) is directed,  there exists some \( \gamma \in A \) such that \( \gamma \geq \alpha \) and \( \gamma \geq \beta \). 

Because both \( (x_\alpha) \) and \( (y_\alpha) \) are increasing nets,  we have:
$$x_\beta + y_\alpha \leq x_\gamma + y_\gamma = z_\gamma \leq u.$$

Therefore,
$x_\beta \leq u - y_\alpha.$

Since this inequality holds for all \( \beta \),  we obtain:
$x \leq u - y_\alpha  \quad \text{for all } \alpha \in A.$

Thus,  \( y_\alpha \leq u - x \) for all \( \alpha \in A \).
Taking the supremum over \( \alpha \),  it follows that:
$y \leq u - x.$

Adding \( y \) to both sides gives:
$x + y \leq u.$

Therefore,  \( x + y \) is less than or equal to any upper bound \( u \) of the net \( (z_\alpha) \),  and since it is also an upper bound itself,  it is the least upper bound. Hence,
$\sup_{\alpha \in A}(x_\alpha + y_\alpha) = x + y.$
\end{proof}
An important consequence of $E$ being an ideal in $E\oplus 
\mathbb{R}
$ is that the Dedekind completeness of $E\oplus 
\mathbb{R}
$ implies the Dedekind completeness of $E$. 
We start our investigation by examining the supremum of a subset of $E$ from
the perspective of it being a subset of $E\oplus 
\mathbb{R}
.$ Details are given in the following lemma.
\begin{lemma}
	Let $E$ be a truncated vector lattice and $A$ a non-empty subset of $E$. If $%
	a_{0}=sup(A)$ exists in $E$, then $a_{0}$ is also the supremum of $A$ when
	considered as a subset of $E\oplus 
	\mathbb{R}
	.$
\end{lemma}
\begin{proof}
	Let $x+\alpha \in E\oplus 
	\mathbb{R}
	$ be an upper bound of $A$.  The goal is to show that $a_{0}\leq x+\alpha $.
	This is clear if $\alpha =0$. Otherwise,  for all $a\in A$,  we have $a\leq
	x+\alpha $.  This implies $\alpha >0$  and  $a-x\leq \alpha $,  which leads to 
	\begin{eqnarray*}
		a &= &x+(a-x)\wedge \alpha \leq x+(a_{0}-x)\wedge \alpha  \\
		&\leq &x+(a_{0}-x)^{+}\wedge \alpha .
	\end{eqnarray*}
	Since $E$ is an ideal in $E\oplus 
	\mathbb{R}
	$ and $(a_{0}-x)^{+}\in E$ it follows that $(a_{0}-x)^{+}\wedge \alpha \in E $.
	Consequently,  $x+(a_{0}-x)^{+}\wedge \alpha $ is an upper bound of $A$ in $%
	E.$ Therefore,  
	\begin{equation*}
		a_{0}\leq x+(a_{0}-x)^{+}\wedge \alpha .
	\end{equation*}
	The latter inequality entails $(a_{0}-x)^{+}=(a_{0}-x)^{+}\wedge \alpha .$
	Hence, 
	\begin{equation*}
		a_{0}=x+a_{0}-x\leq x+(a_{0}-x)^{+}\wedge \alpha \leq x+\alpha .
	\end{equation*}
	completing the proof.
\end{proof}
The next theorem gives us the relationship between the Dedekind completeness of $E$ and $E\oplus 
\mathbb{R}
$ .
\begin{theorem}
	Let $E$ be a truncated vector lattice. Then, $E\oplus 
	\mathbb{R}
	$ is Dedekind complete if and only if $E$ is Dedekind complete and satisfies
	the following property%
	\begin{equation*}
		\left( \ast \right) \text{ Every bouned set in }\bar{E}\text{ has a supremum
			in }E\oplus 
		\mathbb{R}
		.
	\end{equation*}
\end{theorem}

\begin{proof}
	Clearly, if $E\oplus 
	\mathbb{R}
	$ is Dedekind complete, then $E$ is Dedekind complete. Furthermore, 
	property $\left( \ast \right) $ is satisfied since any bounded set in $\bar{E%
	}$ is also bounded in $E\oplus 
	\mathbb{R}
	$ by $1$, and thus its supremum exists in $E\oplus 
	\mathbb{R}
	$ due to the Dedekind completeness of $E\oplus 
	\mathbb{R}
	$.
	
	Conversly, let $\left( x_{i}\right) _{i\in I}$ be a net in $E\oplus 
	\mathbb{R}
	$ such that $0\leq x_{i}\uparrow \leq x+\lambda ,$ where $x+\lambda \in
	E\oplus 
	\mathbb{R}
	.$ This immediately implies $\lambda \geq 0.$
	
	Case $1:$ $\lambda =0.$ In this case $0\leq x_{i}\leq x,$ so the net $\left(
x_{i}\right) _{i\in I}$ is in $E$. By the Dedekind completeness of $E$ and Lemma $ 5.4 $,  we have $0\leq x_{i}\uparrow u$ in $E\oplus 
\mathbb{R}
$ for some $u\in E.$
	
	Case $2:\lambda >0.$ Without loss of generality, we can assume that \(\lambda = 1\). Using the Riesz decomposition property for increasing nets, as established in one of the preceding lemmas,  in the space \(E \oplus \mathbb{R}\), we then obtain the existence of two nets \((a_i)\) and \((b_i)\) such that
$x_i = a_i + b_i, \quad a_i \uparrow \leq |x|, \quad \text{and} \quad b_i \uparrow \leq 1.$

Since \(E\) is Dedekind complete, and the net \((a_i)_{i \in I}\) is bounded in \(E\), the supremum \(a = \sup a_i\) exists in \(E\).

Moreover, each \(b_i\) can be written as \(b_i = c_i + \gamma_i\), where \(c_i \in E\) and \(\gamma_i \in \mathbb{R}\), both forming increasing nets. From the inequality
$0 \leq b_i = c_i + \gamma_i \leq 1,$
it follows that \(\gamma_i \in [0,1]\) and \(c_i \in \overline{E}\).

Then, by the Dedekind completeness of \(\mathbb{R}\) and property \((\ast)\), both suprema \(\gamma = \sup \gamma_i\) and \(c = \sup c_i\) exist.

Finally and since all these nets are increasing, and using the lemma which shows that the supremum is compatible with addition for increasing nets, we conclude that the supremum satisfies:
$\sup x_i = \sup a_i + c_i +  \gamma_i = a + c + \gamma \quad \text{in } E \oplus \mathbb{R}.$
\end{proof}

\begin{corollary}
	Let $E$ be a unital truncated vector lattice. Then $E$ is Dedekind complete
	if and only if $E\oplus 
	\mathbb{R}
	$ is Dedekind complete.
\end{corollary}

\begin{proof}
	We only need to show that property $\left( \ast \right) $ is fulfilled when $%
	E$ is Dedekind complete. 
	
	Assume that $E$ is Dedekind complete. Let $A$ be a non-empty subset of $\bar{%
		E}$, then $A$ is bounded from above by the unit of the truncation. It follow
	that $A$ has a supremum in $E.$
\end{proof}
\section{Lateral Completeness}
Recall that a Riesz space is said to be laterally complete if every disjoint family of positive elements that is bounded above admits a supremum.
We start this section by stating a theorem that will be useful

\begin{theorem}
	\cite[Theorem 7.15]{locallysolid} For an order dense Riesz subspace $E$ of an
	Archimedean Riesz space $M$ we have the following.
	
	If $E$ is laterally complete, then $E$ majorizes $M.$
\end{theorem}

Now, we move on to the statement of the main theorem of this section

\begin{theorem}
	Let $E$ be a truncated Riesz space.
	
	$\left( 1\right) $ If $E$ is laterally complete, then $E\oplus 
	\mathbb{R}
	$ is laterally complete
	
	$\left( 2\right) $ Suppose $E\oplus 
	\mathbb{R}
	$ is laterally complete and Archimedean. Then,  $E$ is laterally complete if
	and only if $E$ is unital
\end{theorem}

\begin{proof}
	$\left( 1\right) $ Let $\left( x_{i}+\alpha _{i}\right) _{i\in I}$ be a
	family of disjoint positive elements in $E\oplus 
	\mathbb{R}
	$.
	
	Case $1:$ Suppose there exist distinct $i,j\in I$ such that $\alpha _{i}>0$
	and $\alpha _{j}>0.$
	
	Since $\left( x_{i}+\alpha _{i}\right) \wedge \left( x_{j}+\alpha
	_{j}\right) =0,$ we have: 
	\begin{equation*}
		0=2\left( x_{i}+\alpha _{i}\wedge x_{j}+\alpha _{j}\right)
		=x_{i}+x_{j}+\alpha _{i}+\alpha _{j}-\mid x_{i}-x_{j}+\alpha _{i}-\alpha
		_{j}\mid .
	\end{equation*}
	
	If $\alpha _{i}=\alpha _{j},$ then $1\in E$,  which is a contradiction.

	Alternatively,  \ if $\alpha _{i}<\alpha _{j},$ we have%
	\begin{equation*}
		0 =x_{i}+x_{j}+\alpha _{i}+\alpha _{j}+2\mid \alpha _{i}-\alpha _{j}\mid 
		\overline{\frac{1}{\alpha _{i}-\alpha _{j}}\left( x_{i}-x_{j}\right)
			^{-}\vee \left( \frac{1}{\alpha _{j}-\alpha _{i}}\left( x_{i}-x_{j}\right)
			^{+}\right) }-\mid \alpha _{i}-\alpha _{j}\mid . \\
		=2\left( x_{i}\wedge x_{j}\right) +2\alpha _{i}+2\mid \alpha _{i}-\alpha
		_{j}\mid \overline{\frac{1}{\alpha _{i}-\alpha _{j}}\left(
			x_{i}-x_{j}\right) ^{-}\vee \left( \frac{1}{\alpha _{j}-\alpha _{i}}\left(
			x_{i}-x_{j}\right) ^{+}\right) }.
	\end{equation*}
	
	Also,  in this case, $1\in E$ and we have a contradiction.
	
	Case $2:$ At most one $\alpha _{i_{0}}>0.$
	
	In this case, for all $i\neq i_{0},$ we have $\alpha _{i}=0,$ so $x_{i}\in E.
	$ Since $E$ is laterally complete,  the supremum $x=\sup \left( x_{i},i\in
	I\diagdown \left\{ i_{0}\right\} \right) $ exists in $E.$ Let $y=x\vee
	\left( x_{i_{0}}+\alpha _{i_{0}}\right) \in E\oplus 
	\mathbb{R}
	.$
	
	Next,  we show that $y=\sup \left( x_{i}+\alpha _{i},i\in I\right) .$
	
	It is clear that,  for all $i\in I,$ $x_{i}+\alpha _{i}\leq y.$ Moreover,  if $%
	z\in E\oplus 
	\mathbb{R}
	$ such that $z\geq x_{i}+\alpha _{i},$ for all $i\in I,$ then for all $i\in
	I\diagdown \left\{ i_{0}\right\} ,$ $z\geq x_{i}.$
	
	Since $x=\sup \left( x_{i},i\in I\diagdown \left\{ i_{0}\right\} \right) \in
	E$ then,  using Lemma 5.4,  it follows that $x\leq z.$ Also,  $z\geq
	x_{i_{0}}+\alpha _{i_{0}}.$ Therefore,  $z\geq x\vee \left( x_{i_{0}}+\alpha
	_{i_{0}}\right) =y.$
	
	Finally,  $x\vee \left( x_{i_{0}}+\alpha _{i_{0}}\right) =\sup \left(
	x_{i}+\alpha _{i},i\in I\right) $ in $E\oplus 
	\mathbb{R}
	.$
	
	$\left( 2\right) $ We suppose that $E\oplus 
	\mathbb{R}
	$ is laterally complete and Archimedean.
	
	If $E$ is unital,  then $E$ is a band in $E\oplus 
	\mathbb{R}
	.$
	Since $E\oplus 
	\mathbb{R}
	$ is laterally complete, any band in it is also laterally complete.  Thus,  $E$
	is laterally complete.
	
	Conversely,  suppose that $E$ is laterally complete.  Assume,  for contradiction,  that $E$ is not unital. 
	
	Because $E$ is dense in $E\oplus 
	\mathbb{R}
	$ in this situation, Theorem 6.1 implies that $E$ is majorizing in $E\oplus 
	\mathbb{R}
	$.  However,  since $E$ is an ideal in $E\oplus 
	\mathbb{R}
	$,  it follows that $E=E\oplus 
	\mathbb{R}
	$,  which is a contradiction. 
	Thus,  we conclude that $E$ is unital.
\end{proof}
\section{Universal Completeness}
In this section, we will investigate the universal completeness of the
unitization $E\oplus 
\mathbb{R}
$. Let us recall that a Riesz space is universally complete if and only if
it is both Dedekind complete and laterally complete.

To facilitate the statement of the main theorem of this section, the
following characterization of Archimedean truncation on a universally
complete vector lattice will be useful.

\begin{lemma}\cite{Boulstructuretheorem}
	Any Archimedean truncation on a universally complete vector lattice is
	unital.
\end{lemma}

The main result of this section is the following:

\begin{theorem}
	Let $E$ be a truncated Riesz space. 
	
	\begin{enumerate}
		\item  Assume that $E$ is universally complete. Then $E\oplus 
		\mathbb{R}
		$ is universally complete if and only if the truncation in $E$ satisfy $%
		\left( \tau _{3}\right) $.
		
		\item Assume that $E\oplus 
		\mathbb{R}
		$ is universally complete. Then $E$ is universally complete if and only if $E
		$ is unital.
	\end{enumerate}
	
	\begin{proof}
		
		\begin{enumerate}
			\item Assume that $E$ is universally complete.
			
			We have already shown that if $E\oplus 
			\mathbb{R}
			$  is Archimedean,  then the truncation
			in E satisfies $\left( \tau _{3}\right) $. Therefore,  the universal
			completeness of $E\oplus 
			\mathbb{R}
			$ implies that condition $\left( \tau _{3}\right) $ is fulfilled.
			Conversely,  since $E$ is universally complete,  if $\left( \tau _{3}\right) $
			is satisfied,  then by the preceding lemma,  $E$ is unital.  Hence,  using
			Corollary 5.6 and Theorem 6.2,  $E\oplus 
			\mathbb{R}
			$ is universally complete.
			
			\item Assume that $E\oplus 
			\mathbb{R}
			$ is universally complete.
			
			In the unital case, as before, the two preceding sections allow us to deduce
			that $E$ is universally complete. Moreover, since $E\oplus 
			\mathbb{R}
			$ is universally complete, then condition $\left( \tau _{3}\right) $ is
			guaranteed, and consequently, thanks to the preceding lemma, $E$ is indeed
			unital.\qquad 
		\end{enumerate}
	\end{proof}
\end{theorem}
\section{Projection Property}
In this section, we study how the projection property transfers between a
truncated Riesz space $E$ and its unitization $E\oplus \mathbb{R}$.

A Riesz space $E$ is said to have the projection property if every band of $E
$ is a projection band. 

We recall the following result which characterizes the ideals that are
projection bands.

\begin{theorem}\cite[Theorem 1.41]{positiveoperator}
	For an ideal $B$ in a Riesz space $E$ the following statements are
	equivalent.
	
	\begin{enumerate}
		\item $B$ is a projection band i.e, $E:B\oplus B^{d}$ holds.
		
		\item For each $x\in E^{+}$ the supremum of the set $B^{+}\cap \left[ 0,x%
		\right] $ exists in $E$ and belongs to $B.$
		
		\item There exists an ideal $A$ of $E$ such that $E=B\oplus A$ holds.
	\end{enumerate}
\end{theorem}

Our result in this section is as follows:

\begin{theorem}
	Let $E$ be a truncated Riesz space.
	
	\begin{enumerate}
		\item If $E\oplus 
		\mathbb{R}
		$ has the projection property than $E$ has the projection property
		
		\item If $E$ is unital and has the projection property then $E\oplus 
		\mathbb{R}
		$ has the projection property
	\end{enumerate}
\end{theorem}

\begin{proof}
	
	\begin{enumerate}
		\item Let $B$ be a band in $E$. Then $B$ is also a band in the larger space $%
		E\oplus 
		\mathbb{R}
		$, and we have the orthogonal decomposition:%
		\begin{equation*}
			E\oplus 
			\mathbb{R}
			=B\oplus B^{d}.
		\end{equation*}
		
		Considering the ideal $B^{\prime }=B^{d}\cap E$ within $E$, we obtain the
		direct sum decomposition: 
		\begin{equation*}
			E=B\oplus B^{\prime }.
		\end{equation*}
		
		This direct sum decomposition demonstrates that $B$ is a projection band in $%
		E$.
		
		\item Let $B$ be a band in $E\oplus 
		\mathbb{R}
		$. Since $E$ is unital, $E$ is a band in $E\oplus 
		\mathbb{R}
		$, and we have the orthogonal decomposition:%
		\begin{equation*}
			E\oplus 
			\mathbb{R}
			=E\oplus E^{d}.
		\end{equation*}
		
		There exist two ideal $B_{1}$ and $B_{2}$ in $E$ and $E^{d}$ respectively
		such that 
		\begin{equation*}
			E=\left( B\cap E\right) \oplus B_{1}\text{ and }E^{d}=\left( B\cap
			E^{d}\right) \oplus B_{2}.
		\end{equation*}
		
		Combining these decompositions, we get:%
		\begin{eqnarray*}
			E\oplus 
			\mathbb{R}
			&=&\left( B\cap E\right) \oplus B_{1}\oplus \left( B\cap E^{d}\right) \oplus
			B_{2} \\
			&=&\left( \left( B\cap E\right) \oplus \left( B\cap E^{d}\right) \right)
			\oplus \left( B_{1}\oplus B_{2}\right) .
		\end{eqnarray*}
		
		Now, let's show that $B=\left( B\cap E\right) \oplus \left( B\cap
		E^{d}\right) .$ To this end, let $x\in B.$ We can decompose $x$ as $%
		x=x^{+}-x^{-}$. Further, we can decompose $x^{+}$ and $x^{-}$ based on their
		components in $E$ and $E^{d}:$%
		\begin{equation*}
			x^{+}=a_{1}+a_{2}\text{ and }x^{-}=b_{1}+b_{2}
		\end{equation*}
		where,

			\begin{itemize}
				\item $a_{1}=\sup \left( z\in E;0\leq z\leq x^{+}\right) \in E$: The
				component of $x^{+}$ in $E$
				
				\item $b_{1}=\sup \left( z\in E^{d};0\leq z\leq x^{-}\right) \in E:$ The
				component of $x^{-}$ in $E$
				
				\item $a_{2}=\sup \left( z\in E;0\leq z\leq x^{+}\right) \in E^{d}$: The
				component of $x^{+}$ in $E^{d}$
				
				\item $b_{2}=\sup \left( z\in E;0\leq z\leq x^{-}\right) \in E^{d}$: The
				component of $x^{-}$ in $E^{d}$
			\end{itemize}

		Since $x^{+}\in B$, $0\leq a_{1}\leq x^{+}$ and $B$ is an ideal, we have $%
		a_{1}\in E\cap B.$ Similarly, $b_{1}\in E\cap B.$ 
		
		Therefore,  the component of $x$ in $E$ is $a_{1}-b_{1}\in B\cap E.$
		
		Analogously,  the component of $x$ in $E^{d}$ is $a_{2}-b_{2}\in B\cap E^{d}.$
		
		Combining these components, we have:%
		\begin{equation*}
			x=a_{1}-b_{1}+a_{2}-b_{2}\in \left( B\cap E\right) \oplus \left( B\cap
			E^{d}\right) 
		\end{equation*}
		
	Hence, $B=\left( B\cap E\right) \oplus \left( B\cap E^{d}\right) $, and
		we conclude that:%
		\begin{equation*}
			E\oplus 
			\mathbb{R}
			=B\oplus \left( B_{1}\oplus B_{2}\right) 
		\end{equation*}
		
		This direct sum decomposition shows that $B$ is a projection band in $%
		E\oplus 
		\mathbb{R}
		$.
	\end{enumerate}
\end{proof}


\bibliographystyle{spmpsci}
\bibliography{biblio}
\end{document}